\documentclass[11pt,reqno]{amsart}
\usepackage{amsmath}
\usepackage{amsthm}
\usepackage{amssymb}
\usepackage[all]{xy}
\usepackage{color}

\newcounter{obs}

\newtheorem{theorem}{Theorem}
\newtheorem{corollary}[theorem]{Corollary}

\newtheorem{proposition}[theorem]{Proposition}
\theoremstyle{definition}
\newtheorem{definition}[theorem]{Definition}
\theoremstyle{remark}
\newtheorem{remark}[theorem]{Remark}
\newtheorem{example}[theorem]{Example}

\numberwithin{equation}{section}
\numberwithin{theorem}{section}

\begin{document}

\title[Strong factorizations of Fourier type representing operators]
{Strong factorizations of operators with applications to Fourier
and Ces\'aro transforms}


\author[O.\ Delgado]{O.\ Delgado}
\address{Departamento de Matem\'atica Aplicada I, E.\ T.\ S.\ de Ingenier\'ia
de Edificaci\'on, Universidad de Sevilla, Avenida de Reina Mercedes,
4 A,  Sevilla 41012, Spain.}
\email{\textcolor[rgb]{0.00,0.00,0.84}{olvido@us.es}}

\author[M.\ Masty{\l}o]{M.\ Masty{\l}o}
\address{Faculty of Mathematics and Computer Science,
Adam Mickiewicz University in Pozna\'n, Umultowska 87, 61-614
Pozna{\'n}, Poland.}
\email{\textcolor[rgb]{0.00,0.00,0.84}{mastylo@amu.edu.pl}}

\author[E.\ A.\ S\'{a}nchez P\'{e}rez]{E.\ A.\ S\'{a}nchez P\'{e}rez}
\address{Instituto Universitario de Matem\'atica Pura y Aplicada,
Universitat Polit\`ecnica de Val\`encia, Camino de Vera s/n, 46022
Valencia, Spain.}
\email{\textcolor[rgb]{0.00,0.00,0.84}{easancpe@mat.upv.es}}



\thanks{The first author gratefully acknowledge the support of the
Ministerio de Econom\'{\i}a y Compe\-ti\-tividad
(project \#MTM2015-65888-C4-1-P) and the Junta de Andaluc\'{\i}a
(project FQM-7276), Spain. The second author was supported by
National Science Centre, Poland, project no. 2015/17/B/ST1/00064.
The third author acknowledges with thanks the support of the
Ministerio de Econom\'{\i}a y Competitividad (project
MTM2016-77054-C2-1-P), Spain.}


\maketitle


\begin{abstract}
Consider two continuous linear operators $T\colon X_1(\mu)\to
Y_1(\nu)$ and $S\colon X_2(\mu)\to Y_2(\nu)$ between Banach
function spaces related to different $\sigma$-finite measures
$\mu$ and $\nu$. We characterize by means of weighted norm
inequalities when $T$ can be strongly factored through $S$, that
is, when there exist functions $g$ and $h$ such that $T(f)=gS(hf)$
for all $f\in X_1(\mu)$. For the case of spaces with Schauder
basis our characterization can be improved, as we show when $S$ is
for instance the Fourier operator, or the Ces\`aro operator. Our
aim is to study the case when the map $T$ is besides injective.
Then we say that it is a~representing operator ---in the sense
that it allows to represent each  elements of the Banach function
space $X(\mu)$ by a~sequence of generalized Fourier
coefficients---, providing a complete characterization of these
maps in terms of weighted norm inequalities. Some examples and
applications involving recent results on the Hausdorff-Young  and
the Hardy-Littlewood inequalities for operators on weighted Banach
function spaces are also provided.
\end{abstract}



\section{Introduction}

Let $X_1(\mu)$, $X_2(\mu)$, $Y_1(\nu)$, $Y_2(\nu)$ be Banach
function spaces related to different $\sigma$-finite measures
$\mu$, $\nu$ and consider two continuous linear operators $T\colon
X_1(\mu)\to Y_1(\nu)$, $S\colon X_2(\mu)\to Y_2(\nu)$. In this paper we provide
a~characterization in terms
of weighted norm inequalities of when $T$ can be factored through $S$
via multiplication operators, that is, when there are functions
$g$~and $h$ satisfying that $T(f)=gS(hf)$ for all $f\in X_1(\mu)$.

This problem was studied in \cite{delgado-sanchezperez2} for the
case when $\mu$ and $\nu$ are the same finite measure.
However, the results developed there do not allow to face the problem we study here,
in which different $\sigma$-finite  measures $\mu$ and $\nu$ appear in order to consider
the relevant case of the classical sequence spaces $\ell^p$. The reason is that
we are interested in considering standard cases as the Fourier and the Ces\`aro operators,
that will be in fact our main examples.

In this direction, we will show that in the case when the K\"{o}the dual $Y_1(\nu)'$ of $Y(\nu)$ and
$X_1(\mu)$ have Schauder basis, the norm inequality which
characterizes the factorization of $T$ through $S$ can be weakened. After showing this, we will
develop with some detail some the examples regarding Fourier operators, operators factoring though
infinite matrices and the Ces\`aro operator. This will allow to introduce the notion of
\textit{representing operator} and to study it in the second part of the paper.

Let us explain  briefly this notion. With the notation introduced above,
assume that $Y_1(\nu)$ and $Y_2(\nu)$ have  unconditional basis $\mathcal U_1:=\{v_i: i \in \mathbb N \}$ and $\mathcal U_2:=\{e_i: i \in \mathbb N \}$, respectively.
 Suppose that there exists a Schauder basis $\mathcal B:\{f_i: \in \mathcal N \}$ for the space $X_2(\mu)$ and write $\alpha_i(f)$ for the $i$-th basic coefficient of $f \in X_2(\mu)$, $i \in \mathbb N$.
 We will say that an operator $T: X_1(\mu) \to Y_1(\nu)$ is a \textit{representing operator} for $X_1(\mu)$ on $Y_1(\nu)$ (associated to the basis $\mathcal B$ of $X_2(\mu)$) if \textit{each element $x \in X_1(\mu)$ can be represented univocally by a sequence of coefficients $(\beta_i(x))$ such that $ \sum_{i=1}^\infty \beta_i(x) v_i \in Y_1(\nu)$, where the coefficients $\beta_i(x)$ can be computed by means of the associated  values of $\alpha_i$ by a simple transformation provided by multiplication operators.}

Thus, the last part of the paper is devoted to find a characterization of such operators in terms of vector norm inequalities that they must satisfy. We  provide also classical and recently published
examples of such kind of maps, using for instance an improvement of the Hausdorff-Young inequality given in \cite{kellogg}, or the continuity of the Fourier operator $\mathcal H_p: L^p[- \pi, \pi] \to \ell^p(W)$ ---where $\ell^p(W)$ is a weighted $\ell^p$-space--- that can be found in \cite{ash}.


\section{Preliminaries}\label{SEC: Preliminaries}

Let $(\Omega,\Sigma,\mu)$ be a $\sigma$-finite measure space and
denote by $L^0(\mu)$ the space of all measurable real functions
defined on $\Omega$, where functions which are equal $\mu$-a.e.\
are identified. By a~\emph{Banach function space} we mean a~Banach space $X(\mu)\subset L^0(\mu)$
with norm $\Vert\cdot\Vert_{X}$ satisfying that if $f\in
L^0(\mu)$, $g\in X(\mu)$ and $|f|\le|g|$ $\mu$-a.e.\ then $f\in
X(\mu)$ and $\Vert f\Vert_{X}\le\Vert g\Vert_{X}$. In particular
$X(\mu)$ is a Banach lattice for the $\mu$-a.e.\ pointwise order,
in which the convergence in norm of a~sequence implies the
convergence $\mu$-a.e.\ for some subsequence. Note that every
positive linear operator between Banach lattices is continuous,
(see \cite[p.\,2]{lindenstrauss-tzafriri}). So, all inclusions
between  Banach function
spaces are continuous. General information about Banach function spaces can be found for
instance in \cite[Ch.\,15]{zaanen} considering the function norm
$\rho$ defined there as $\rho(f)=\Vert f\Vert_{X}$ if $f\in
X(\mu)$ and $\rho(f)=\infty$ in other case.

A Banach function space
$X(\mu)$ is said to be \emph{saturated} if there is no
$A\in\Sigma$ with $\mu(A)>0$ such that $f\chi_A=0$ $\mu$-a.e.\ for
all $f\in X(\mu)$. This is equivalent to the existence of
a~function $g\in X(\mu)$ such that $g>0$ $\mu$-a.e.

Given two  Banach function
spaces $X(\mu)$ and $Y(\mu)$, the \emph{$Y(\mu)$-dual space} of
$X(\mu)$ is defined by
$$
X^Y=\big\{h\in L^0(\mu):\, fh\in Y(\mu) \textnormal{ for all } f\in X(\mu)\big\}.
$$
Every $h\in X^Y$ defines a continuous multiplication operator
$M_h\colon X(\mu)\to Y(\mu)$ via $M_h(f)=fh$ for all $f\in
X(\mu)$. The space $X^Y$ is a
Banach function space with norm
$$
\Vert h\Vert_{X^Y}=\sup_{f\in B_X}\Vert hf\Vert_Y, \quad\, h\in
X^Y,
$$
if and only if $X(\mu)$ is saturated. As usual $B_X$ denotes the closed unit ball
of $X(\mu)$. Note that $X^{L^1}$ is just the classical K\"{o}the dual space $X(\mu)'$ of $X(\mu)$.
If $X(\mu)$ is saturated then $X(\mu)'$ is also saturated. This does not hold in general for $X^Y$.
For issues related to generalized dual spaces see  \cite{calabuig-delgado-sanchezperez} and the references therein.

A saturated  Banach function
space $X(\mu)$ is  contained in its
K\"{o}the bidual $X(\mu)''$ with $\Vert f\Vert_{X''}\le\Vert
f\Vert_X$ for all $f\in X(\mu)$. It is known that $\Vert
f\Vert_{X''}=\Vert f\Vert_X$ for all $f\in X(\mu)$ if and only if
$X(\mu)$ is \emph{order semi-continuous}, that is, if for every
$f,f_n\in X(\mu)$ such that $0\le f_n\uparrow f$ $\mu$-a.e.\ it
follows that $\Vert f_n\Vert_X\uparrow\Vert f\Vert_X$. Even more,
$X(\mu)=X(\mu)''$ with equal norms if and only if $X(\mu)$ has the
\emph{Fatou property}, that is, if for every $f_n\in X(\mu)$ such
that $0\le f_n\uparrow f$ $\mu$-a.e.\ and $\sup_n\Vert
f_n\Vert_{X}<\infty$, we have that $f\in X(\mu)$ and $\Vert
f_n\Vert_{X}\uparrow\Vert f\Vert_{X}$.

Denote by $X(\mu)^*$ the topological dual of a~saturated
 Banach function space
$X(\mu)$. Every function $h\in X(\mu)'$ defines an element
$\eta(h)\in X(\mu)^*$ via $\langle \eta(h),f\rangle=\int hf\,d\nu$
for all $f\in X(\mu)$. The map $\eta\colon X(\mu)'\to X(\mu)^*$ is
a continuous linear injection, since the norm of every $h\in
X(\mu)'$ can be computed as
$$
\Vert h\Vert_{X'}=\sup_{f\in B_X}\Big|\int hf\,d\mu\Big|
$$
and so $\eta$ is an isometry. It is known that $\eta$ is surjective if and only if
$X(\mu)$ is \emph{$\sigma$-order continuous}, that is, if for every $(f_n)\subset
X(\mu)$ with $f_n\downarrow0$ $\mu$-a.e.\ it follows that $\Vert f_n\Vert_{X}\downarrow0$.
Note that $\sigma$-order continuity implies order semi-continuity.

The \emph{$\sigma$-order continuous part} $X_a(\mu)$ of a~
 saturated Banach function
space $X(\mu)$ is the largest $\sigma$-order continuous closed
solid subspace of $X(\mu)$, which can be described as
$$
X_a(\mu)=\big\{f\in X(\mu):\, |f|\ge f_n\downarrow0
\textnormal{$\mu$-a.e.\ implies } \Vert
f_n\Vert_X\downarrow0\big\}.
$$
Also, a function $f\in X_a(\mu)$ if and only if $f\in X(\mu)$
satisfies that $\Vert f\chi_{A_n}\Vert_X\downarrow0$ whenever
$(A_n)\subset\Sigma$ is such that $A_n\downarrow$ with $\mu(\cap
A_n)=0$. Note that $X_a(\mu)$ could be the trivial space as in the
case of $X(\mu)=L^\infty(\mu)$ when $\mu$ is nonatomic.
 In the case when $X_a(\mu)$ is saturated,
$X_a(\mu)$ is order dense in $L^0(\mu)$ and so by the Monotone
Convergence Theorem, it follows easily that $X_a(\mu)'=X(\mu)'$
with equal norms.

The \emph{$\pi$-product space} $X\pi Y$ of two  Banach function spaces $X(\mu)$ and
$Y(\mu)$ is defined as the space of functions $h\in L^0(\mu)$ such
that $|h|\le\sum_n|f_ng_n|$ $\mu$-a.e.\ for some sequences
$(f_n)\subset X(\mu)$ and $(g_n)\subset Y(\mu)$ satisfying
$\sum_n\Vert f_n\Vert_X\Vert g_n\Vert_Y<\infty$. For $h\in X\pi
Y$, consider the norm
$$
\pi(h)=\inf\Big\{\sum_n\Vert f_n\Vert_X\Vert g_n\Vert_Y\Big\},
$$
where the infimum is taken over all sequences $(f_n)\subset
X(\mu)$ and $(g_n)\subset Y(\mu)$ such that
$|h|\le\sum_n|f_ng_n|$ $\mu$-a.e. and
$\sum_n\Vert f_n\Vert_X\Vert g_n\Vert_Y<\infty$. The space $X\pi
Y$ is a saturated Banach
function space with norm $\pi$ if and only if $X(\mu)$, $Y(\mu)$
and $X^{Y'}$ are saturated and, in this case, $(X\pi Y)'=X^{Y'}$
with equal norms (see \cite[Proposition2.2]{delgado-sanchezperez1}).  The calculus of product spaces is nowadays well-known (see \cite{calabuig-delgado-sanchezperez,delgado-sanchezperez1,kolesmal,schepmultip}); the reader can find all the information that is needed on this construction in these papers.


 Banach function spaces on the measure space
$(\mathbb{N},\mathcal{P}(\mathbb{N}),\lambda)$ with the counting
measure $\lambda$ are are called usually Banach sequence spaces.
The classical Banach sequence spaces $\ell^p$ for $1\le
p\le\infty$ is saturated which is $\sigma$-order continuous if and
only if $p<\infty$. As usual for each $n\in\mathbb{N}$, we denote
by $(e^n)$ the standard unit vector basis in $c_0$.

We recall the well known easily verified formula $
(\ell^p)^{\ell^q}=\ell^{s_{pq}}$ with equal norms, where
\begin{equation}\label{EQ: spq-Def}
1\le s_{pq}=\left\{\begin{array}{ll}
\frac{pq}{p-q} & \textnormal{if } 1\le q<p<\infty \\ q & \textnormal{if } 1\le q<p=\infty \\
\infty & \textnormal{if } 1\le p\le q\le\infty
\end{array}\right..
\end{equation}
In particular, $(\ell^p)'=(\ell^p)^{\ell^1}=\ell^{p'}$ where $p'$ denote the conjugate exponent of
$p$ ($\frac{1}{p}+\frac{1}{p'}=1$). Note that $\ell^p$ has the Fatou property as $(\ell^p)''=\ell^p$.
Also note that $s_{pq}=1$ if and only if $q=1$ and $p=\infty$.


\section{Strong factorization of operators on Banach
function spaces}\label{SEC: StrongFactor}

Let $(\Omega,\Sigma,\mu)$, $(\Delta, \Gamma, \nu)$ be
$\sigma$-finite measure spaces, $X_1(\mu)$, $X_2(\mu)$,
$Y_1(\nu)$, $Y_2(\nu)$ saturated
 Banach function spaces and $T\colon X_1(\mu)\to
Y_1(\nu)$, $S\colon X_2(\mu)\to Y_2(\nu)$ nontrivial continuous
linear operators. For $h\in X_1^{X_2}$, we will say that $T$
\emph{factors strongly} through $S$ and $M_h$ if there exists
$g\in Y_2^{Y_1''}$ such that the diagram
\begin{equation}\label{EQ: StrongFactor}
\begin{split}
\xymatrix{
X_1(\mu) \ \ar[r]^{T} \ar@{.>}[d]_{M_h} & \ Y_1(\nu) \ \ar@{.>}[r]^{i} & \ Y_1(\nu)'' \\
X_2(\mu) \ \ar[r]^{S} & \ Y_2(\nu) \ \ar@{.>}[ru]_{M_g}}
\end{split}
\end{equation}
commutes. Here $i$ denotes the inclusion map. Note that if $Y_1(\nu)$ has the Fatou property
the diagram above looks as
$$
\xymatrix{
X_1(\mu) \ \ar[r]^{T} \ar@{.>}[d]_{M_h} & \ Y_1(\nu) \\
X_2(\mu) \ \ar[r]^{S} & \ Y_2(\nu) \ \ar@{.>}[u]_{M_g}}
$$
In the case when $\mu$ and $\nu$ are the same finite measure and under certain extra conditions,
\cite[Theorem 4.1]{delgado-sanchezperez2} characterizes when $T$ factors strongly through $S$ and $M_h$.
In this section we extend this theorem to our more general setting and improve it by relaxing the conditions.
The extension will be obtained from the following broader factorization result.

\begin{theorem}\label{THM: GeneralFactor}
Assume that $Y_2^{Y_1''}$ is saturated and consider a function
$h\in X_1^{X_2}$. The following statements are equivalent{\rm:}
\begin{itemize}\setlength{\leftskip}{-3ex}\setlength{\itemsep}{.5ex}
\item[(a)] There exists a constant $C>0$ such that the inequality
$$
\sum_{i=1}^n\int T(x_i)y_i'\,d\nu\le
C\Big\Vert\sum_{i=1}^nS(hx_i)y_i'\Big\Vert_{Y_2\pi Y_1'}, \quad\,
n\in\mathbb{N},
$$
holds for every  $n\in\mathbb{N}$, $x_1,...,x_n\in
X_1(\mu)$ and $y_1',...,y_n'\in Y_1(\nu)'$.

\item[(b)] There exists $\xi^*\in(Y_2\pi Y_1')^*$ satisfying the
following factorization between the operators $T$ and $S${\rm:}
$$
\xymatrix{
X_1(\mu) \ \ar[r]^{T} \ar@{.>}[d]_{M_h} & \ Y_1(\nu) \ \ar@{.>}[r]^{i}
& \ Y_1(\nu)'' \ \ar@{.>}[r]^{\eta} & \ Y_1(\nu)'^* \\
X_2(\mu) \ \ar[r]^{S} & \ Y_2(\nu) \ \ar@{.>}[rru]_{R_{\xi^*}}}
$$
where $\eta$ is the continuous linear injection of $Y_1(\nu)''$ into $Y_1(\nu)'^*$ and $R_{\xi^*}$ is the continuous
linear operator defined by $\langle R_{\xi^*}(y_2),y_1'\rangle=\langle \xi^*,y_2y_1'\rangle$ for $y_2\in Y_2(\nu)$
and $y_1'\in Y_1(\nu)'$.
\end{itemize}
\end{theorem}

\begin{proof}
 Note that the condition of $Y_2^{Y_1''}$ being saturated assures that $Y_2\pi Y_1'$ is a~saturated
  Banach function space. Also note that the map $R_{\xi^*}\colon Y_2(\nu)\to Y_1(\nu)'^*$
 defined in (b) is a well defined continuous linear operator as
$$
|\langle R_{\xi^*}(y_2),y_1'\rangle|\le\Vert\xi^*\Vert_{(Y_2\pi Y_1')^*}\Vert y_2y_1'\Vert_{Y_2\pi Y_1'}\le
\Vert\xi^*\Vert_{(Y_2\pi Y_1')^*}\Vert y_2\Vert_{Y_2}\Vert y_1'\Vert_{Y_1'}
$$
for all $y_2\in Y_2(\nu)$ and $y_1'\in Y_1(\nu)'$.

(a) $\Rightarrow$ (b) For every $n\in\mathbb{N}$, $x_1,...,x_n\in X_1(\mu)$ and
$y_1',...,y_n'\in Y_1(\nu)'$ we take the convex function $\phi\colon B_{(Y_2\pi Y_1')^*}\to\mathbb{R}$ given by
$$
\phi(\xi^*)=\sum_{i=1}^n\int T(x_i)y_i'\,d\nu-C\sum_{i=1}^n\langle \xi^*,S(hx_i)y_i'\rangle
$$
for all $\xi^*\in B_{(Y_2\pi Y_1')^*}$. Considering the weak*
topology on $(Y_2\pi Y_1')^*$ we have that $\phi$ is a~continuous
map on a~compact convex set. Moreover, from the Hahn-Banach
Theorem there exists $\xi_\phi^*\in B_{(Y_2\pi Y_1')^*}$ such that
$$
\Big\Vert\sum_{i=1}^nS(hx_i)y_i'\Big\Vert_{Y_2\pi Y_1'}=\Big\langle \xi_\phi^*,\sum_{i=1}^nS(hx_i)y_i'\Big\rangle
$$
and so, by (a), it follows that $\phi(\xi_\phi^*)\le0$.

Since the family $\mathcal{F}$ of functions $\phi$ defined in this way is concave,
Ky Fan's lemma (see for instance \cite[E.4]{pietsch}) guarantees the existence of an element
$\xi^*\in B_{(Y_2\pi Y_1')^*}$ such that $\phi(\xi^*)\le0$ for all $\phi\in\mathcal{F}$.
In particular, for every $x\in X_1(\mu)$ and
$y'\in Y_1(\nu)'$, we have that
$$
\int T(x)y'\,d\nu\le C\langle \xi^*,S(hx)y'\rangle.
$$
By taking $-y'$ instead of $y'$, we obtain that
$$
-\int T(x)y'\,d\nu\le-C\langle \xi^*,S(hx)y'\rangle
$$
and so
$$
\langle\eta\big(T(x)\big),y'\rangle=\langle R_{C\xi^*}\big(S(hx)\big),y'\rangle.
$$
Therefore, $\eta\big(T(x)\big)=R_{C\xi^*}\big(S(hx)\big)$ for all $x\in X_1(\mu)$ and
the factorization in (b) holds for $C\xi^*\in(Y_2\pi Y_1')^*$.

(b) $\Rightarrow$ (a)  For each $n\in\mathbb{N}$ and every
$x_1,...,x_n\in X_1(\mu)$, $y_1',...,y_n'\in Y_1(\nu)'$ we have
that

\begin{eqnarray*}
\sum_{i=1}^n\int T(x_i)y_i'\,d\nu & = & \sum_{i=1}^n\langle\eta\big(T(x_i)\big),y_i'\rangle=
\sum_{i=1}^n\langle R_{\xi^*}\big(S(hx_i)\big),y_i'\rangle \\ & = &
\sum_{i=1}^n\langle \xi^*,S(hx_i)y_i'\rangle = \Big\langle \xi^*,\sum_{i=1}^nS(hx_i)y_i'\Big\rangle
\\ & \le & \Vert\xi^*\Vert_{(Y_2\pi Y_1')^*}\Big\Vert\sum_{i=1}^nS(hx_i)y_i'\Big\Vert_{Y_2\pi Y_1'}.
\end{eqnarray*}
Note that $\Vert\xi^*\Vert_{(Y_2\pi Y_1')^*}>0$ as $T$ is nontrivial.
\end{proof}

Note that the condition of $Y_2^{Y_1''}$ being saturated is
obtained for instance if $Y_2(\nu)\subset Y_1(\nu)''$ which is
equivalent to $L^\infty(\nu)\subset Y_2^{Y_1''}$. Also note that
the condition (a) of Theorem \ref{THM: GeneralFactor} is
equivalent to
$$
\left|\sum_{i=1}^n\int T(x_i)y_i'\,d\nu\right|\le
C\Big\Vert\sum_{i=1}^nS(hx_i)y_i'\Big\Vert_{Y_2\pi Y_1'}, \quad\,
n\in\mathbb{N}
$$
 for every
$x_1,...,x_n\in X_1(\mu)$ and $y_1',...,y_n'\in Y_1(\nu)'$.
Indeed, we only have to take $-y_1',...,-y_n'$ instead of
$y_1',...,y_n'$ in Theorem \ref{THM: GeneralFactor}.(a).

As a consequence of Theorem \ref{THM: GeneralFactor}, we obtain the following generalization and
improvement of \cite[Theorem 4.1]{delgado-sanchezperez2}.

\begin{corollary}\label{COR: StrongFactor}
Assume that $Y_2^{Y_1''}$ is saturated and that $y_2y_1'\in
(Y_2\pi Y_1')_a$ for all $y_2\in Y_2(\nu)$ and $y_1'\in
Y_1(\nu)'$. Given $h\in X_1^{X_2}$, the following statements are
equivalent{\rm:}
\begin{itemize}\setlength{\leftskip}{-3ex}\setlength{\itemsep}{.5ex}
\item[(a)] $T$ factors strongly through $S$
and $M_h$.

\item[(b)] There exists a constant $C>0$ such that the inequality
$$
\sum_{i=1}^n\int T(x_i)y_i'\,d\nu\le
C\Big\Vert\sum_{i=1}^nS(hx_i)y_i'\Big\Vert_{Y_2\pi Y_1'}, \quad\,
n\in\mathbb{N}
$$
holds for all every $x_1,...,x_n\in X_1(\mu)$ and
$y_1',...,y_n'\in Y_1(\nu)'$.
\end{itemize}
\end{corollary}

\begin{proof}
First note that $(Y_2\pi Y_1')_a$ is saturated. Indeed, by taking $0<y_2\in Y_2(\nu)$ and $0<y_1'\in Y_1(\nu)'$
we have that $0<y_2y_1'\in (Y_2\pi Y_1')_a$. Then,
$$
(Y_2\pi Y_1')_a'=(Y_2\pi Y_1')'=Y_2^{Y_1''}.
$$
(b) $\Rightarrow$ (a)  From Theorem \ref{THM: GeneralFactor} there exists $\xi^*\in(Y_2\pi Y_1')^*$ such that
$$
\langle\eta\big(T(x)\big),y'\rangle=\langle R_{\xi^*}\big(S(hx)\big),y'\rangle=\langle\xi^*,S(hx)y'\rangle
$$
for all $x\in X_1(\mu)$ and $y'\in Y_1(\nu)'$.
Denote by $\widetilde{\xi}^*$ the restriction of $\xi^*$ to $(Y_2\pi Y_1')_a$.
Since $(Y_2\pi Y_1')_a$ is $\sigma$-order continuous and $\widetilde{\xi}^*\in(Y_2\pi Y_1')_a^*$, we can identify
$\widetilde{\xi}^*$ with a function $g\in(Y_2\pi Y_1')_a'=Y_2^{Y_1''}$, that is, $\langle\widetilde{\xi}^*,
z\rangle=\int gz\,d\nu$ for all $z\in(Y_2\pi Y_1')_a$.
Then, for every $x\in X_1(\mu)$ and $y'\in Y_1(\nu)'$, we have that
$$
\langle\eta\big(T(x)\big),y'\rangle=\langle\xi^*,S(hx)y'\rangle=\langle\widetilde{\xi}^*,S(hx)y'\rangle=\int gS(hx)y'\,
d\nu=\langle\eta\big(gS(hx)\big),y'\rangle
$$
and so $T(x)=gS(hx)$.

(a) $\Rightarrow$ (b) Let $g\in Y_2^{Y_1''}=(Y_2\pi Y_1')'$ be such that $T(x)=gS(hx)$ for all $x\in X_1(\mu)$.
Consider the continuous linear injection  $\widetilde{\eta}\colon (Y_2\pi Y_1')'\to (Y_2\pi Y_1')^*$. Then
$\widetilde{\eta}(g)\in(Y_2\pi Y_1')^*$ satisfies
\begin{eqnarray*}
\langle R_{\widetilde{\eta}(g)}\big(S(hx)\big),y'\rangle &
= & \langle\widetilde{\eta}(g),S(hx)y'\rangle=\int gS(hx)y'\,d\mu \\ & = & \int T(x)y'\,d\mu=
\langle\eta\big(T(x)\big),y'\rangle
\end{eqnarray*}
for all $x\in X_1(\mu)$ and $y'\in Y_1(\nu)'$ and so Theorem
\ref{THM: GeneralFactor}.(b) holds for $\xi^*=\widetilde{\eta}(g)$.
\end{proof}

\begin{remark}\label{REM: StrongFactor}
Of course, the condition $y_2y_1'\in (Y_2\pi Y_1')_a$ for all $y_2\in Y_2(\nu)$ and $y_1'\in Y_1(\nu)'$
holds when $Y_2\pi Y_1'$ is $\sigma$-order continuous. But also this condition is obtained for instance if
any of $Y_2(\nu)$ or $Y_1(\nu)'$ is $\sigma$-order continuous. Indeed, suppose that $Y_2(\nu)$ is $\sigma$-order
continuous and take $y_2\in Y_2(\nu)=(Y_2)_a(\nu)$ and $y_1'\in Y_1(\nu)'$. For every $(A_n)\subset\Sigma$ such
that $A_n\downarrow$ with $\nu(\cap A_n)=0$, we have that
$$
\Vert y_2y_1'\chi_{A_n}\Vert_{Y_2\pi Y_1'}\le\Vert y_2\chi_{A_n}\Vert_{Y_2}\cdot\Vert y_1'\Vert_{Y_1'}\to0
$$
and so $y_2y_1'\in(Y_2\pi Y_1')_a$. We get the case when $Y_1(\nu)'$ is  $\sigma$-order continuous in a similar way.

\end{remark}

\begin{remark}\label{REM: StrongFactor2}
Note that if the $\sigma$-order continuous part $X_a(\mu)$ of
a~saturated Banach function
space $X(\mu)$ is also saturated then $\Vert x\Vert_X=\Vert
x\Vert_{X''}$ for all $x\in X_a(\mu)$. Indeed, for every $x\in
X_a(\mu)$ we have that $\Vert x\Vert_X=\Vert x\Vert_{X_a''}$ since
$X_a(\mu)$ is order semi-continuous and $\Vert
x\Vert_{X_a''}=\Vert x\Vert_{X''}$ since $X_a(\mu)'=X(\mu)'$ with
equal norms. Then, the norm in Corollary \ref{COR:
StrongFactor}.(b) can be computed as
\begin{eqnarray*}
\Big\Vert\sum_{i=1}^nS(hx_i)y_i'\Big\Vert_{Y_2\pi Y_1'}
& = & \Big\Vert\sum_{i=1}^nS(hx_i)y_i'\Big\Vert_{(Y_2\pi Y_1')''}=
\Big\Vert\sum_{i=1}^nS(hx_i)y_i'\Big\Vert_{\big(Y_2^{Y_1''}\big)'}
\\ & = & \sup_{f\in B_{Y_2^{Y_1''}}}\int\Big|f\sum_{i=1}^nS(hx_i)y_i'\Big|\,d\nu.
\end{eqnarray*}
\end{remark}


\section{Strong factorization involving Schauder basis}\label{SEC: StrongFactor}

Let $(\Omega,\Sigma,\mu)$, $(\Delta, \Gamma, \nu)$ be
$\sigma$-finite measure spaces, $X_1(\mu)$, $X_2(\mu)$,
$Y_1(\nu)$, $Y_2(\nu)$ saturated  Banach function spaces and $T\colon X_1(\mu)\to
Y_1(\nu)$, $S\colon X_2(\mu)\to Y_2(\nu)$ nontrivial continuous
linear operators. In this section we assume the existence of
a~Schauder basis $(\gamma_n)$ for $Y_1(\nu)'$ and denote by
$(\gamma_n^*)$ the sequence of coefficient functionals with
respect to this basis.

\begin{theorem}\label{THM: SchauderStrongFactor}
Assume that $Y_2^{Y_1''}$ is saturated and that any of $Y_2(\nu)$
or $Y_1(\nu)'$ is $\sigma$-order continuous. Given $h\in
X_1^{X_2}$, the following statements are equivalent{\rm:}
\begin{itemize}\setlength{\leftskip}{-3ex}\setlength{\itemsep}{.5ex}
\item[(a)] $T$ factors strongly through $S$ and $M_h$.

\item[(b)] There exists a constant $C>0$ such that the inequality
$$
\sum_{i=1}^n\int T(x_i)\gamma_i\,d\nu\le
C\Big\Vert\sum_{i=1}^nS(hx_i)\gamma_i\Big\Vert_{Y_2\pi Y_1'},
\quad\, n\in\mathbb{N}
$$
holds for every $x_1,...,x_n\in
X_1(\mu)$.
\end{itemize}
Moreover, if $Y_2(\nu)\subset Y_1(\nu)''$ and the functions $(\gamma_n)$ have pairwise
disjoint support, then the condition
\begin{itemize}\setlength{\leftskip}{-3ex}\setlength{\itemsep}{.5ex}
\item[(c)] There exists a constant $C>0$ such that the inequality
$$
\int T(x)\gamma_n\,d\nu\le C\int|S(hx)\gamma_n|\,d\nu, \quad\,
n\in \mathbb{N}
$$
holds for every  $x\in X_1(\mu)$ and
$n\ge1$.
\end{itemize}
implies (a)-(b). In the case when $Y_2^{Y_1''}=L^\infty(\nu)$, we have that
(c) is equivalent to (a)-(b).
\end{theorem}

\begin{proof}
(a) $\Leftrightarrow$ (b) From Remark \ref{REM: StrongFactor}, we only have to prove that the condition
(b) of the present theorem implies the condition (b) of Corollary \ref{COR: StrongFactor}.
The converse implication follows by taking $y_i'=\gamma_i$. Let $x_1,...,x_n\in X_1(\mu)$
and $y_1',...,y_n'\in Y_1(\nu)'$. Fix $m\in\mathbb{N}$ and denote $(y_i')^m=\sum_{k=1}^m\langle\gamma_k^*,y_i'\rangle\,\gamma_k$.
It follows that
\begin{eqnarray}\label{EQ: SchauderStrongFactor}
\sum_{i=1}^n\int T(x_i)(y_i')^m\,d\nu & = &
\sum_{i=1}^n\sum_{k=1}^m\langle\gamma_k^*,y_i'\rangle\int
T(x_i)\gamma_k\,d\nu \nonumber \\
& = & \sum_{k=1}^m\int
\Big(\sum_{i=1}^n\langle\gamma_k^*,y_i'\rangle\,T(x_i)\Big)\gamma_k\,d\nu
\nonumber \\ & = &
\sum_{k=1}^m\int T\Big(\sum_{i=1}^n\langle\gamma_k^*,y_i'\rangle\,x_i\Big)\gamma_k\,d\nu \nonumber \\
& \le &
C\Big\Vert\sum_{k=1}^mS\Big(h\sum_{i=1}^n\langle\gamma_k^*,y_i'\rangle\,x_i\Big)\gamma_k\Big\Vert_{Y_2\pi Y_1'} \nonumber \\ & = &
C\Big\Vert\sum_{k=1}^m\sum_{i=1}^n\langle\gamma_k^*,y_i'\rangle\,S(hx_i)\gamma_k\Big\Vert_{Y_2\pi Y_1'}
\nonumber \\ & = &
C\Big\Vert\sum_{i=1}^nS(hx_i)(y_i')^m\Big\Vert_{Y_2\pi Y_1'}.
\end{eqnarray}
Since $(y_i')^m\to y_i'$ in $Y_1(\nu)'$ as $m\to\infty$ and
$$
\Big|\int zy_i'\,d\nu-\int z(y_i')^m\,d\nu\Big|=\Big|\int z\big(y_i'-(y_i')^m\big)\,d\nu\Big|\le\Vert z\Vert_{Y_1}\Vert y_i'-(y_i')^m\Vert_{Y_1'}
$$
for every $z\in Y_1(\nu)$, we have that $\sum_{i=1}^n\int T(x_i)(y_i')^m\,d\nu\to\sum_{i=1}^n\int T(x_i)y_i'\,d\nu$
as $m\to\infty$. On other hand, since
$$
\Vert zy_i'-z(y_i')^m\Vert_{Y_2\pi Y_1'}=\Vert z\big(y_i'-(y_i')^m\big)\Vert_{Y_2\pi Y_1'}\le\Vert z\Vert_{Y_2}\Vert y_i'-(y_i')^m\Vert_{Y_1'}
$$
for every $z\in Y_2(\nu)$, we have that $\sum_{i=1}^nS(hx_i)(y_i')^m\to\sum_{i=1}^nS(hx_i)y_i'$ in $Y_2\pi Y_1'$ as $m\to\infty$.
Then, taking limit as $m\to\infty$ in \eqref{EQ: SchauderStrongFactor}, we obtain
$$
\sum_{i=1}^n\int T(x_i)y_i'\,d\nu\le C\Big\Vert\sum_{i=1}^nS(hx_i)y_i'\Big\Vert_{Y_2\pi Y_1'}.
$$

Assume that $Y_2(\nu)\subset Y_1(\nu)''$ and that the functions $(\gamma_n)$ have pairwise disjoint support. Let us see that (c) implies (b).
The condition $Y_2(\nu)\subset Y_1(\nu)''$ is equivalent to $L^\infty(\nu)\subset Y_2^{Y_1''}=\big(Y_2\pi Y_1'\big)'$ and so
$Y_2\pi Y_1'\subset\big(Y_2\pi Y_1'\big)''\subset L^\infty(\nu)'=L^1(\nu)$. Denote by $K$ the continuity constant of the inclusion
$Y_2\pi Y_1'\subset L^1(\nu)$. For every $n\in\mathbb{N}$ and
$x_1,...,x_n\in X_1(\mu)$, noting that $\sum_{i=1}^n|S(hx_i)\gamma_i|=\big|\sum_{i=1}^nS(hx_i)\gamma_i\big|$ pointwise (as $(\gamma_k)$
have disjoint support), we have that
\begin{eqnarray*}
\sum_{i=1}^n\int T(x_i)\gamma_i\,d\nu & \le & C\sum_{i=1}^n\int|S(hx_i)\gamma_i|\,d\nu=C\int\Big|\sum_{i=1}^nS(hx_i)\gamma_i\Big|\,d\nu \\ & \le &
CK\Big\Vert \sum_{i=1}^nS(hx_i)\gamma_i\Big\Vert_{Y_2\pi Y_1'}.
\end{eqnarray*}
If moreover $L^\infty(\nu)=Y_2^{Y_1''}$ then (a) implies (c), as if $g\in Y_2^{Y_1''}$ is such that $T(x)=gS(hx)$ for all $x\in X_1(\mu)$,
it follows that
$$
\int T(x)\gamma_n\,d\nu=\int gS(hx)\gamma_n\,d\nu\le\int|gS(hx)\gamma_n|\,d\nu\le \Vert g\Vert_\infty\int |S(hx)\gamma_n|\,d\nu.
$$
\end{proof}

Now suppose that there is also a~Schauder basis $(\beta_n)$ for
$X_1(\mu)$ and denote by $(\beta_n^*)$ the sequence of its
coefficient functionals. Then, the equivalent inequalities for the
strong factorization can be relaxed.

\begin{theorem}\label{THM: SchauderStrongFactor2}
Assume that $Y_2^{Y_1''}$ is saturated and that any of $Y_2(\nu)$
or $Y_1(\nu)'$ is $\sigma$-order continuous. Given $h\in
X_1^{X_2}$, the following statements are equivalent{\rm:}
\begin{itemize}\setlength{\leftskip}{-3ex}\setlength{\itemsep}{.5ex}
\item[(a)] $T$ factors strongly through $S$ and $M_h$.

\item[(b)] There exists $g\in Y_2^{Y_1''}$ such that
$T(\beta_n)=gS(h\beta_n)$ for each $n\in\mathbb{N}$.

\item[(c)] There exists a constant $C>0$ such that the inequality
$$
\sum_{i=1}^n\sum_{j=1}^mr_{ij}\int T(\beta_j)\gamma_i\,d\nu\le
C\Big\Vert\sum_{i=1}^n\sum_{j=1}^mr_{ij}S(h\beta_j)\gamma_i\Big\Vert_{Y_2\pi
Y_1'}, \quad\, n,m\in\mathbb{N}
$$
holds for every
$(r_{ij})\subset B_{\ell^\infty}$.
\end{itemize}
Moreover, if $Y_2(\nu)\subset Y_1(\nu)''$  and the functions $(\gamma_n)$ have pairwise disjoint support,
then the condition
\begin{itemize}\setlength{\leftskip}{-3ex}\setlength{\itemsep}{.5ex}
\item[(d)] There exists a constant $C>0$ such that the inequality
$$
\sum_{j=1}^mr_j\int T(\beta_j)\gamma_n\,d\nu\le C\int\Big|\sum_{j=1}^mr_jS(h\beta_j)\gamma_n\Big|\,d\nu
$$
holds for every  $n,m\in\mathbb{N}$ and $(r_j)\subset B_{\ell^\infty}$.
\end{itemize}
implies (a)-(c). In the case when $L^\infty(\nu)=Y_2^{Y_1''}$, we have that (d) is equivalent to (a)-(c).
\end{theorem}

\begin{proof}
(a) $\Rightarrow$ (b) Let $g\in Y_2^{Y_1''}$ be such that $T(x)=gS(hx)$ for all $x\in X_1(\mu)$. In particular,
for $x=\beta_n$ we obtain (b).

(b) $\Rightarrow$ (c) Since $g\in Y_2^{Y_1''}=(Y_2\pi Y_1')'$, for every $n,m\in\mathbb{N}$ and $(r_{ij})\subset B_{\ell^\infty}$,
it follows that
\begin{eqnarray*}
\sum_{i=1}^n\sum_{j=1}^mr_{ij}\int T(\beta_j)\gamma_i\,d\nu & = & \sum_{i=1}^n\sum_{j=1}^mr_{ij}\int gS(h\beta_j)\gamma_i\,d\nu \\
& = & \int g\sum_{i=1}^n\sum_{j=1}^mr_{ij} S(h\beta_j)\gamma_i\,d\nu \\
& \le & \int \Big|g\sum_{i=1}^n\sum_{j=1}^mr_{ij} S(h\beta_j)\gamma_i\Big|\,d\nu \\
& \le & \Vert g\Vert_{(Y_2\pi Y_1')'}\Big\Vert\sum_{i=1}^n\sum_{j=1}^mr_{ij}S(h\beta_j)\gamma_i\Big\Vert_{Y_2\pi Y_1'}.
\end{eqnarray*}

(c) $\Rightarrow$ (a)
Let us show that the condition (b) of Theorem \ref{THM: SchauderStrongFactor} holds.
Let $x_1,...,x_n\in X_1(\mu)$ which can be assumed to be non-null. Fix $m\in\mathbb{N}$ large enough such
that $(x_i)^m=\sum_{j=1}^m\langle\beta_j^*,x_i\rangle\,\beta_j\not=0$ and denote
$\alpha=\max_{i=1,..,n \atop j=1,...,m}|\langle\beta_j^*,x_i\rangle|
$. By taking
$r_{ij}=\frac{\langle\beta_j^*,x_i\rangle}{\alpha}$ it follows that
\begin{eqnarray}\label{EQ: SchauderStrongFactor2}
\sum_{i=1}^n\int T\big((x_i)^m\big)\gamma_i\,d\nu & = & \sum_{i=1}^n\sum_{j=1}^m\langle\beta_j^*,x_i\rangle\int T(\beta_j)\gamma_i\,d\nu
\nonumber \\ & = & \alpha\sum_{i=1}^n\sum_{j=1}^mr_{ij}\int T(\beta_j)\gamma_i\,d\nu
\nonumber \\ & \le &
\alpha\,C\,\Big\Vert\sum_{i=1}^n\sum_{j=1}^mr_{ij}S(h\beta_j)\gamma_i\Big\Vert_{Y_2\pi Y_1'}
\nonumber \\ & = &
C\,\Big\Vert\sum_{i=1}^n\sum_{j=1}^m\langle\beta_j^*,x_i\rangle S(h\beta_j)\gamma_i\Big\Vert_{Y_2\pi Y_1'} \nonumber \\ & = &
C\,\Big\Vert\sum_{i=1}^nS\big(h(x_i)^m\big)\gamma_i\Big\Vert_{Y_2\pi Y_1'}.
\end{eqnarray}
Denoting by $\Vert T \Vert$ the operator norm of $T$, since $(x_i)^m\to x_i$ in $X_1(\mu)$ as $m\to\infty$ and
\begin{eqnarray*}
\Big|\int T(x_i)z\,d\nu-\int T\big((x_i)^m\big)z\,d\nu\Big| & = & \Big|\int T\big(x_i-(x_i)^m\big)z\,d\nu\Big| \\ & \le &
\Vert z\Vert_{Y_1'}\,\Vert T\big(x_i-(x_i)^m\big)\Vert_{Y_1} \\ & \le & \Vert z\Vert_{Y_1'}\,\Vert T\Vert\,\Vert x_i-(x_i)^m\Vert_{X_1}
\end{eqnarray*}
for every $z\in Y_1(\nu)'$, we have that $\sum_{i=1}^n\int T\big((x_i)^m\big)\gamma_i\,d\nu\to\sum_{i=1}^n\int T(x_i)\gamma_i\,d\nu$
as $m\to\infty$. On other hand, denoting by $\Vert S \Vert$ the operator norm of $S$, since
\begin{eqnarray*}
\Vert S(hx_i)z-S\big(h(x_i)^m\big)z\Vert_{Y_2\pi Y_1'} & = & \Vert S\big(h(x_i-(x_i)^m)\big)z\Vert_{Y_2\pi Y_1'}
\\ & \le & \Vert z\Vert_{Y_1'}\,\Vert  S\big(h(x_i-(x_i)^m)\big)\Vert_{Y_2} \\ & \le & \Vert z\Vert_{Y_1'}\,
\Vert S\Vert\,\Vert h(x_i-(x_i)^m)\Vert_{X_2}
\\ & \le & \Vert z\Vert_{Y_1'}\,\Vert S\Vert\,\Vert h\Vert_{X_1^{X_2}}\,\Vert x_i-(x_i)^m\Vert_{X_1}
\end{eqnarray*}
for every $z\in Y_1(\nu)'$, we have that $\sum_{i=1}^nS\big(h(x_i)^m\big)\gamma_i\to\sum_{i=1}^nS\big(hx_i\big)\gamma_i$ in $Y_2\pi Y_1'$
as $m\to\infty$.
Then, Taking limit as $m\to\infty$ in \eqref{EQ: SchauderStrongFactor2}, we obtain
$$
\sum_{i=1}^n\int T(x_i)\gamma_i\,d\nu\le C\,\Big\Vert\sum_{i=1}^nS(hx_i)\gamma_i\Big\Vert_{Y_2\pi Y_1'}.
$$

Assume that $Y_2(\nu)\subset Y_1(\nu)''$ and that the functions $(\gamma_n)$ have pairwise disjoint support.
We have already noted that in this case $Y_2\pi Y_1'\subset L^1(\nu)$ (denote by $K$ its continuity constant)
and $\sum_{i=1}^n|f_i\gamma_i|=\big|\sum_{i=1}^nf_i\gamma_i\big|$ pointwise for every $n\in\mathbb{N}$ and
$(f_i)\subset L^0(\nu)$. Let us see that (d) implies (c). For every $n,m\in\mathbb{N}$ and $(r_{ij})\subset B_{\ell^\infty}$
we have that
\begin{eqnarray*}
\sum_{i=1}^n\sum_{j=1}^mr_{ij}\int T(\beta_j)\gamma_i\,d\nu & \le & C\sum_{i=1}^n\int\Big|\sum_{j=1}^mr_{ij}S(h\beta_j)\gamma_i\Big|\,d\nu \\ & = &
C\int\Big|\sum_{i=1}^n\sum_{j=1}^mr_{ij}S(h\beta_j)\gamma_i\Big|\,d\nu \\ & \le &
CK\Big\Vert\sum_{i=1}^n\sum_{j=1}^mr_{ij}S(h\beta_j)\gamma_i\Big\Vert_{Y_2\pi Y_1'}.
\end{eqnarray*}
If moreover $L^\infty(\nu)=Y_2^{Y_1''}$ then (a) implies (d), as if $g\in Y_2^{Y_1''}$ is such that $T(x)=gS(hx)$
for all $x\in X_1(\mu)$, it follows that
\begin{eqnarray*}
\sum_{j=1}^mr_j\int T(\beta_j)\gamma_n\,d\nu & = & \sum_{j=1}^mr_j\int gS(h\beta_j)\gamma_n\,d\nu=\int g\sum_{j=1}^mr_jS(h\beta_j)\gamma_n\,d\nu \\
& \le & \int \Big|g\sum_{j=1}^mr_jS(h\beta_j)\gamma_n\Big|\,d\nu\le \Vert g\Vert_\infty\int \Big|\sum_{j=1}^mr_jS(h\beta_j)\gamma_n\Big|\,d\nu.
\end{eqnarray*}
\end{proof}


\section{Examples: the Fourier and Ces\`aro operators}

In this section we show how the results obtained in the previous one can be applied in concrete contexts. In particular, we will deal with the Fourier operator acting in different weighted $L^p$-spaces, we will show factorization through infinite matrices and, as a special case, we will analyze the case provided by the Ces\`aro operator.

\subsection{Strong factorization through the Fourier operator}\label{SEC: FourierOperator}

Consider the measure space given by the interval $\mathbb{T}=[-\pi,\pi]$,
its Borel $\sigma$-algebra and the Lebesgue measure $m$ and denote by $(\phi_n)$
the real trigonometric system on $\mathbb{T}$, that is,
$$
\phi_n(x)=\left\{\renewcommand{\arraystretch}{2}\begin{array}{ll}
{\displaystyle\frac{1}{\sqrt{2\pi}}} & \textnormal{if } \, n=1 \\
{\displaystyle\frac{\cos(kx)}{\sqrt{\pi}}} & \textnormal{if } \, n=2k \\
{\displaystyle\frac{\sin(kx)}{\sqrt{\pi}}} & \textnormal{if } \, n=2k+1
\end{array}\right..
$$
Note that $\int_{-\pi}^\pi \phi_i(x)\phi_j(x)\,dx=0$ if $i\not=j$ and $\int_{-\pi}^\pi \phi_i(x)\phi_i(x)\,dx=1$.
Each function $f\in L^1(m)$ is associated to its Fourier series $S(f)=\sum_{n\ge1}a_n\phi_n$
where $a_n=\int_\mathbb{T} f\phi_n\,dm$. If $f\in L^r(m)$ for $1<r<\infty$ then $S(f)$ converges to $f$ in $L^r(m)$
and so $(\phi_n)$ is a Schauder basis on $L^r(m)$.

Let $\mathcal{F}$ be the Fourier operator  defined by
\[
\mathcal{F}(f)=\Big(\int_\mathbb{T} f\phi_n\,dm\Big), \quad\, f\in
L^1(m).
\]
 The
Hausdorff-Young inequality (see for instance
\cite[(8.5.7)]{hardy-littlewood-polya}) guarantees that
$$
\mathcal{F}\colon L^r(m)\to\ell^{r'}
$$
is a well defined
continuous operator for every $1<r\le2$.

Fix $1<r\le2$, $r\le p<\infty$, $1<q\le\infty$ and let $T\colon L^p(m)\to\ell^q$ be a non-trivial continuous linear operator.
We have that $(\phi_n)$ is a Schauder basis for $L^p(m)$ (as $1<p<\infty$) and $(e^n)$ is a Schauder basis for $(\ell^q)'$ (as $q>1$).
Also, $L^p(m)\subset L^r(m)$ (as $r\le p$) and so $\chi_\mathbb{T}\in (L^p)^{L^r}$.

\begin{proposition}\label{PROP: FourierFactor}
The following statements are equivalent{\rm:}
\begin{itemize}\setlength{\leftskip}{-3ex}\setlength{\itemsep}{.5ex}
\item[(a)] $T$ factors strongly through $\mathcal{F}$, that is, there exists $g\in\ell^{s_{r'\hspace{-.5mm}q}}$ such that
$$
\xymatrix{
L^p(m) \ \ar@{.>}[d]_{i} \ar[r]^(.6){T} & \ \ell^q  \\
L^r(m) \ \ar[r]^(0.6){\mathcal{F}} & \ \ell^{r'} \ \ar@{.>}[u]_{M_g}}
$$
$($see \eqref{EQ: spq-Def} in the preliminaries for the definition
of $s_{r'\hspace{-.5mm}q})$.

\item[(b)] $T(\phi_n)_i=0$ for all $i\not=n$ and $\big(T(\phi_i)_i\big)\in\ell^{s_{r'\hspace{-.5mm}q}}$.

\item[(c)] There exists a constant $C>0$ such that the inequality
$$
\sum_{i=1}^n\sum_{j=1}^mr_{ij}T(\phi_j)_i \le
C\Big(\sum_{i=1}^{\min\{n,m\}}|r_{ii}|^{s_{r'\hspace{-.5mm}q}'}\Big)^{\frac{1}{s_{r'\hspace{-.5mm}q}'}},
\quad\, n,m\in\mathbb{N}
$$
holds for every
$(r_{ij})\subset B_{\ell^\infty}$.
\end{itemize}
Moreover, in the case when $r'\le q$, the conditions (a)-(c) are equivalent to
\begin{itemize}\setlength{\leftskip}{-3ex}\setlength{\itemsep}{.5ex}
\item[(d)] There exists a constant $C>0$ such that the inequality
$$
\sum_{j=1}^mr_jT(\phi_j)_n\le C\left\{\begin{array}{ll}
|r_n| & \textnormal{if } n\le m \\ 0 & \textnormal{if } n>m
\end{array}\right.
$$
holds for  each
$n,m\in\mathbb{N}$ and all $(r_j)\subset
B_{\ell^\infty}$.
\end{itemize}
\end{proposition}

\begin{proof}
Note that both $\ell^{r'}$ and $(\ell^q)'$ are $\sigma$-order continuous (as $r,q>1$)
and that $(\ell^{r'})^{(\ell^q)''}=(\ell^{r'})^{\ell^q}=\ell^{s_{r'\hspace{-.5mm}q}}$
where  $s_{r'\hspace{-.5mm}q}$ is defined as in \eqref{EQ: spq-Def}.
For the equivalence among (a), (b) and (c), let us see that conditions (b) and (c) are
just respectively conditions (b) and (c) of Theorem \ref{THM: SchauderStrongFactor2}
rewritten for $X_1(\mu)=L^p(m)$, $X_2(\mu)=L^r(m)$, $Y_1(\nu)=\ell^q$ ($\nu$ being the
counting measure $\lambda$ on $\mathbb{N}$), $Y_2(\nu)=\ell^{r'}$, $S=\mathcal{F}$,
$h=\chi_\mathbb{T}$, $(\beta_n)=(\phi_n)$ and $(\gamma_n)=(e^n)$.

(b) $\Rightarrow$ Theorem \ref{THM: SchauderStrongFactor2}.(b).
Take $g=\big(T(\phi_i)_i\big)\in\ell^{s_{r'\hspace{-.5mm}q}}$.
Then, for every $n,i\in\mathbb{N}$ we have that $T(\phi_n)_i=T(\phi_i)_i\mathcal{F}(\phi_n)_i=g_i\mathcal{F}(\phi_n)_i$ and so
$T(\phi_n)=g\mathcal{F}(\phi_n)$.

Theorem \ref{THM: SchauderStrongFactor2}.(b) $\Rightarrow$ (b). Let $g\in\ell^{s_{r'\hspace{-.5mm}q}}$ be such that
$T(\phi_n)=g\mathcal{F}(\phi_n)$ for all $n\in\mathbb{N}$. Then
$$
T(\phi_n)_i=g_i\mathcal{F}(\phi_n)_i=\left\{\begin{array}{ll} g_i & \textnormal{if } i=n \\ 0 & \textnormal{if } i\not=n
\end{array}\right.
$$
and so $\big(T(\phi_i)_i\big)=g\in\ell^{s_{r'\hspace{-.5mm}q}}$.

(c) $\Leftrightarrow$ Theorem \ref{THM:
SchauderStrongFactor2}.(c). From Remark \ref{REM: StrongFactor2}
and noting that $\big(\ell^{r'}\pi(\ell^q)'\big)''
=\big((\ell^{r'})^{(\ell^q)''}\big)'=(\ell^{s_{r'\hspace{-.5mm}q}})'
=\ell^{s_{r'\hspace{-.5mm}q}'}$ with equals norms and
$s_{r'\hspace{-.5mm}q}'<\infty$ (as $s_{r'\hspace{-.5mm}q}>1$),
for  each  $n,m\in\mathbb{N}$
and  all $(r_{ij})\subset B_{\ell^\infty}$ it
follows that
\begin{eqnarray*}
\Big\Vert\sum_{i=1}^n\sum_{j=1}^mr_{ij}\mathcal{F}(\phi_j)e^i\Big\Vert_{\ell^{r'}\pi(\ell^q)'} &
= & \Big\Vert\sum_{i=1}^n\sum_{j=1}^mr_{ij}\mathcal{F}(\phi_j)e^i\Big\Vert_{(\ell^{r'}\pi(\ell^q)')''} \\
& = & \Big\Vert\sum_{i=1}^n\sum_{j=1}^mr_{ij}\mathcal{F}(\phi_j)e^i\Big\Vert_{\ell^{s_{r'\hspace{-.5mm}q}'}}
\\ & = & \Big\Vert\sum_{i=1}^n\sum_{j=1}^mr_{ij}\mathcal{F}(\phi_j)_ie^i\Big\Vert_{\ell^{s_{r'\hspace{-.5mm}q}'}} \\
& = & \Big(\sum_{i=1}^n\Big|\sum_{j=1}^mr_{ij}\mathcal{F}(\phi_j)_i\Big|^{s_{r'\hspace{-.5mm}q}'}\Big)^{\frac{1}{s_{r'\hspace{-.5mm}q}'}}
\\ & = & \Big(\sum_{i=1}^{\min\{n,m\}}|r_{ii}|^{s_{r'\hspace{-.5mm}q}'}\Big)^{\frac{1}{s_{r'\hspace{-.5mm}q}'}}
\end{eqnarray*}
and
$$
\sum_{i=1}^n\sum_{j=1}^mr_{ij}\int T(\phi_j)e^i\,d\lambda=\sum_{i=1}^n\sum_{j=1}^mr_{ij}T(\phi_j)_i.
$$

In the case when $r'\le q$ we have that $s_{r'\hspace{-.5mm}q}=\infty$ and so $(\ell^{r'})^{(\ell^q)''}=\ell^\infty$.
Then (d) is equivalent to (a)-(c) as (d) is to rewrite condition (d) of Theorem \ref{THM: SchauderStrongFactor2}. Indeed,
$$
\sum_{j=1}^mr_j\int T(\phi_j)e^n\,d\lambda=\sum_{j=1}^mr_jT(\phi_j)_n
$$
and
$$
\int\Big|\sum_{j=1}^mr_j\mathcal{F}(\phi_j)e^n\Big|\,d\lambda=\Big|\sum_{j=1}^mr_j\mathcal{F}(\phi_j)_n\Big|=\left\{\begin{array}{ll}
|r_n| & \textnormal{if } n\le m \\ 0 & \textnormal{if } n>m
\end{array}\right..
$$

\end{proof}


\subsection{Strong factorization for infinite matrices and the Ces\`aro operator}\label{SEC: Matrix}

Consider the measure space
$(\mathbb{N},\mathcal{P}(\mathbb{N}),\lambda)$ with $\lambda$
being the counting measure on $\mathbb{N}$. Let $X_1(\lambda)$,
$X_2(\lambda)$, $Y_1(\lambda)$, $Y_2(\lambda)$ be saturated
 Banach function spaces in
which $(e^n)$ is a~Schauder basis and $T\colon X_1(\lambda)\to
Y_1(\lambda)$, $S\colon X_2(\lambda)\to Y_2(\lambda)$ be
nontrivial continuous linear operators. Then, the operators $T$
and $S$ can be described by infinite matrices $(a_{ij})$ and
$(b_{ij})$ respectively, namely $a_{ij}=T(e^j)_i$ and
$b_{ij}=S(e^j)_i$. We also require that $(e^n)$ is a Schauder
basis for $Y_1(\lambda)'$.

\begin{proposition}\label{PROP: MatrixStrongFactor}
Assume that $Y_2^{Y_1''}$ is saturated and that any of
$Y_2(\lambda)$ or $Y_1(\lambda)'$ is $\sigma$-order continuous.
Given $h\in X_1^{X_2}$, the following statements are
equivalent{\rm:}
\begin{itemize}\setlength{\leftskip}{-3ex}\setlength{\itemsep}{.5ex}
\item[(a)] $T$ factors strongly through $S$ and $M_h$.

\item[(b)] There exists $g\in Y_2^{Y_1''}$ such that $\frac{a_{ij}}{b_{ij}}=g_ih_j$
whenever $b_{ij}\not=0$ and $a_{ij}=0$ whenever $b_{ij}=0$.

\item[(c)] There exists a constant $C>0$ such that the inequality
$$
\sum_{i=1}^n\sum_{j=1}^mr_{ij}a_{ij} \le
C\Big\Vert\sum_{i=1}^n\Big(\sum_{j=1}^mh_jr_{ij}b_{ij}\Big)e^i\Big\Vert_{Y_2\pi
Y_1'}, \quad\, n,m\in\mathbb{N}
$$
holds for every
$(r_{ij})\subset B_{\ell^\infty}$.
\end{itemize}
Moreover, if $Y_2(\lambda)\subset Y_1(\lambda)''$, then the condition
\begin{itemize}\setlength{\leftskip}{-3ex}\setlength{\itemsep}{.5ex}
\item[(d)] There exists a constant $C>0$ such that the inequality
$$
\sum_{j=1}^mr_ja_{nj}\le C\Big|\sum_{j=1}^mh_jr_jb_{nj}\Big|,
\quad\, n,m\in\mathbb{N}
$$
holds for every
$(r_j)\subset B_{\ell^\infty}$,
\end{itemize}
implies (a)-(c). In the case when $\ell^\infty=Y_2^{Y_1''}$, we have that (d) is equivalent to (a)-(c).
\end{proposition}

\begin{proof}
We only have to see that conditions (b), (c) and (d) are just respectively conditions (b),
(c) and (d) of Theorem \ref{THM: SchauderStrongFactor2} rewritten for $\mu=\nu$ being the counting measure
$\lambda$ and $(\beta_n)=(\gamma_n)=(e^n)$.
Note that for every $i,j\in\mathbb{N}$ we have that
$a_{ij}=T(e^j)_i$ and $g_ih_jb_{ij}=g_ih_jS(e^j)_i=g_iS(h_je^j)_i=g_iS(he^j)_i$. So (b) $\Leftrightarrow$
Theorem \ref{THM: SchauderStrongFactor2}.(b). Since $\int T(e^j)e^i\,d\lambda=T(e^j)_i=a_{ij}$ and
$S(he^j)e^i=S(h_je^j)e^i=h_jS(e^j)e^i=h_jS(e^j)_ie^i=h_jb_{ij}e^i$ we have that (c) $\Leftrightarrow$
Theorem \ref{THM: SchauderStrongFactor2}.(c). Moreover as
$$
\int\Big|\sum_{j=1}^mr_jS(he^j)e^n\Big|\,d\lambda=\int\Big|\sum_{j=1}^mr_jh_jb_{nj}e^n\Big|\,d\lambda=\Big|\sum_{j=1}^mr_jh_jb_{nj}\Big|,
$$
it follows that (d) $\Leftrightarrow$ Theorem \ref{THM: SchauderStrongFactor2}.(d).
\end{proof}



Let $\mathcal{C}$ be the Ces\`aro operator which maps a~real sequence $x=(x_n)$ into the
sequence of its Ces\`aro means
$\mathcal{C}(x)=\big(\frac{1}{n}\sum_{i=1}^nx_i\big)$. It is well
known that $\mathcal{C}\colon\ell^r\to\ell^r$ continuously for
every $1<r<\infty$ (see \cite[Theorem
326]{hardy-littlewood-polya}) and it can be described by the
infinite matrix $(b_{ij})$ where $b_{ij}=\frac{1}{i}$ if $j\le i$
and $b_{ij}=0$ if $j>i$, that is,
$$
(b_{ij})=\left(\begin{array}{cccccc} 1 & 0 & 0 & 0 & 0 & \cdots \\
\frac{1}{2} & \frac{1}{2} & 0 & 0 & 0 & \cdots \\
\frac{1}{3} & \frac{1}{3} & \frac{1}{3} & 0 & 0 & \cdots \\
\frac{1}{4} & \frac{1}{4} & \frac{1}{4} & \frac{1}{4} & 0 & \cdots \\
\vdots & \vdots & \vdots & \vdots & \vdots & \ddots \end{array}\right).
$$

Fix $1\le p<\infty$, $1<q,r<\infty$ and let $T\colon \ell^p\to\ell^q$ be a nontrivial continuous operator
described by the infinite matrix $(a_{ij})$  with $a_{ij}=T(e^j)_i$. Note that
$(e^n)$ is a Schauder Basis on $\ell^p$, $\ell^q$, $\ell^r$ and $(\ell^q)'$.

\begin{proposition}\label{PROP: CesaroFactor}
Let $h\in\ell^{s_{pr}}$ (see \eqref{EQ: spq-Def} for the
definition of $s_{pr}$). The following statements are
equivalent{\rm:}
\begin{itemize}\setlength{\leftskip}{-3ex}\setlength{\itemsep}{.5ex}
\item[(a)] $T$ factors strongly through $\mathcal{C}$ and $M_h$, that is,
there exists $g\in\ell^{s_{rq}}$ such that
$$
\xymatrix{
\ell^p \ \ar[r]^{T} \ar@{.>}[d]_{M_h} & \ \ell^q  \\
\ell^r \ \ar[r]^{\mathcal{C}} & \ \ell^r \ \ar@{.>}[u]_{M_g}}
$$

\item[(b)] There exists $g\in\ell^{s_{rq}}$ such that
$$
a_{ij}=\left\{\begin{array}{ll} \frac{g_ih_j}{i} & \textnormal{if } j\le i \\ 0 & \textnormal{if } j>i \end{array}\right.
$$

\item[(c)] There exists a constant $C>0$ such that the inequality
$$
\sum_{i=1}^n\sum_{j=1}^m r_{ij}a_{ij}\le C
\Big(\sum_{i=1}^n\frac{1}{i^{s_{rq}'}}\Big|\sum_{j=1}^{\min\{i,m\}}h_jr_{ij}\Big|^{s_{rq}'}\Big)^{\frac{1}{s_{rq}'}},
\quad\, n,m\in\mathbb{N}
$$
holds for every  $(r_{ij})\in
B_{\ell^\infty}$.
\end{itemize}
Moreover, in the case when $r\le q$, the conditions (a)-(c) are equivalent to
\begin{itemize}\setlength{\leftskip}{-3ex}\setlength{\itemsep}{.5ex}
\item[(d)] There exists a~constant $C>0$ such that the inequality
$$
\sum_{j=1}^mr_ja_{nj}\le
C\frac{1}{n}\Big|\sum_{j=1}^{\min\{n,m\}}h_jr_j\Big|, \quad\,
n,m\in\mathbb{N}
$$
holds for every  $(r_j)\subset
B_{\ell^\infty}$.
\end{itemize}
\end{proposition}

\begin{proof}
Note that both $\ell^r$ and $(\ell^q)'$ are $\sigma$-order continuous
(as $r<\infty$ and $q>1$), $(\ell^r)^{(\ell^q)''}=(\ell^r)^{\ell^q}=\ell^{s_{rq}}$ and $(\ell^p)^{\ell^r}=\ell^{s_{pr}}$.
Also note that if $r\le q$ then $s_{rq}=\infty$ and so $(\ell^r)^{(\ell^q)''}=\ell^\infty$.
Then, we only have to see that (b), (c), (d) is just to rewrite respectively conditions (b), (c),
(d) of Proposition \ref{PROP: MatrixStrongFactor} for $X_1(\lambda)=\ell^p$, $X_2(\lambda)=\ell^r$,
$Y_1(\lambda)=\ell^q$, $Y_2(\lambda)=\ell^r$ and $S=\mathcal{C}$.

As noted above, the elements of the matrix of $\mathcal{C}$ are $b_{ij}=\frac{1}{i}$ if $j\le i$ and $b_{ij}=0$ if
$j>i$, so (b) $\Leftrightarrow$ Proposition \ref{PROP: MatrixStrongFactor}.(b).

By Remark \ref{REM: StrongFactor2} and noting that $\big(\ell^r\pi(\ell^q)'\big)''
=\big((\ell^r)^{(\ell^q)''}\big)'=(\ell^{s_{r\hspace{-.5mm}q}})'
=\ell^{s_{r\hspace{-.5mm}q}'}$ with equals norms and $s_{r\hspace{-.5mm}q}'<\infty$ (as $s^{r\hspace{-.5mm}q}>1$),
for every $n,m\in\mathbb{N}$ and $(r_j)\subset B_{\ell^\infty}$ it follows that
\begin{eqnarray*}
\Big\Vert\sum_{i=1}^n\Big(\sum_{j=1}^mh_jr_{ij}b_{ij}\Big)e^i\Big\Vert_{\ell^r\pi(\ell^q)'}
& = & \Big\Vert\sum_{i=1}^n\Big(\sum_{j=1}^mh_jr_{ij}b_{ij}\Big)e^i\Big\Vert_{(\ell^r\pi(\ell^q)')''} \\
& = & \Big\Vert\sum_{i=1}^n\Big(\sum_{j=1}^mh_jr_{ij}b_{jj}\Big)e^i\Big\Vert_{\ell^{s_{rq}'}}
\\ & = & \Big(\sum_{i=1}^n\Big|\sum_{j=1}^mh_jr_{ij}b_{ij}\Big|^{s_{rq}'}\Big)^{\frac{1}{s_{rq}'}}
\\ & = & \Big(\sum_{i=1}^n\frac{1}{i^{s_{rq}'}}\Big|\sum_{j=1}^{\min\{i,m\}}h_jr_{ij}\Big|^{s_{rq}'}\Big)^{\frac{1}{s_{rq}'}}.
\end{eqnarray*}
Hence, (c) $\Leftrightarrow$ Proposition \ref{PROP: MatrixStrongFactor}.(c).

(d) $\Leftrightarrow$ Proposition \ref{PROP: MatrixStrongFactor}.(d) holds as
$$
\sum_{j=1}^mh_jr_jb_{nj}=\frac{1}{n}\sum_{j=1}^{\min\{n,m\}}h_jr_j.
$$
\end{proof}

Finally we show how the matrix of $T$ must looks for $T$ can be strongly factored through the Ces\`aro operator.

\begin{proposition}\label{PROP: CesaroFactor2}
Let $h\in\ell^{s_{pr}}$ and suppose that $h_1\not=0$. The following statements are equivalent:
\begin{itemize}\setlength{\leftskip}{-3ex}\setlength{\itemsep}{.5ex}
\item[(a)] $T$ factors strongly through $\mathcal{C}$ and $M_h$.

\item[(b)] $a_{ij}=0$ for $j>i$, $a_{ij}=\frac{h_ja_{i1}}{h_1}$ for $j\le i$ and $(ia_{i1})\in\ell^{s_{rq}}$.

\item[(c)] The matrix of $T$ looks as
$$
\left(\begin{array}{cccccc}
h_1\alpha_1 & 0 & 0 & 0 & 0 & \cdots \\
h_1\alpha_2 & h_2\alpha_2 & 0 & 0 & 0 & \cdots \\
h_1\alpha_3 & h_2\alpha_3 & h_3\alpha_3 & 0 & 0 & \cdots \\
h_1\alpha_4 & h_2\alpha_4 & h_3\alpha_4 & h_4\alpha_4 & 0 & \cdots \\
h_1\alpha_5 & h_2\alpha_5 & h_3\alpha_5 & h_4\alpha_5 & h_5\alpha_5 & \cdots \\
\vdots & \vdots & \vdots & \vdots & \vdots & \ddots\end{array}\right)
$$
where $(\alpha_n)\in\mathbb{R}^{\mathbb{N}}$ is such that
$(n\alpha_n)\in\ell^{s_{rq}}$.
\end{itemize}
\end{proposition}

\begin{proof}
(a) $\Rightarrow$ (b)
From Proposition \ref{PROP: CesaroFactor} there exists $g\in\ell^{s_{rq}}$ such that
$$
a_{ij}=\left\{\begin{array}{ll} \frac{g_ih_j}{i} & \textnormal{if } j\le i \\ 0 & \textnormal{if } j>i \end{array}\right..
$$
Then $a_{i1}=\frac{g_ih_1}{i}$ for all $i$ and so $a_{ij}=\frac{h_ja_{i1}}{h_1}$ for every $j\le i$.
Also note that $(ia_{i1})=h_1g\in\ell^{s_{rq}}$.

(b) $\Rightarrow$ (c) Taking
$(\alpha_n)=\big(\frac{a_{n1}}{h_1}\big)$ we have that
$h_j\alpha_i=\frac{h_ja_{i1}}{h_1}=a_{ij}$ for every $j\le i$ and
$(n\alpha_n)=\frac{1}{h_1}(na_{n1})\in\ell^{s_{rq}}$.

(c) $\Rightarrow$ (a) Taking $g=(i\alpha_i)\in\ell^{s_{rq}}$ it follows that
$$
a_{ij}=\left\{\begin{array}{ll} h_j\alpha_i=\frac{g_ih_j}{i} & \textnormal{if } j\le i \\ 0 & \textnormal{if } j>i \end{array}\right..
$$
Then, from Proposition \ref{PROP: CesaroFactor}, (a) holds.
\end{proof}

If $T$ factors strongly through $\mathcal{C}$ and $M_h$ then there exists $h_j\not=0$ as $T$ is non trivial.
So, given $0\not=h\in\ell^{s_{pr}}$ and denoting $j_0=\min\big\{j\in\mathbb{N}: h_j\not=0\big\}$,
similarly to Proposition \ref{PROP: CesaroFactor2}, we have that $T$ factors strongly through $\mathcal{C}$
and $M_h$ if and only if its matrix looks as
$$
\left(\begin{array}{cccccccc}
0 & \cdots & 0 & 0 & 0 & 0 & 0 & \cdots \\
\vdots &  & \vdots & \vdots & \vdots & \vdots & \vdots &  \\
0 & \cdots & 0 & 0 & 0 & 0 & 0 & \cdots \\
0 & \cdots & 0 & h_{j_0}\alpha_1 & 0 & 0 & 0 & \cdots \\
0 & \cdots & 0 & h_{j_0}\alpha_2 & h_{j_0+1}\alpha_2 & 0 & 0 & \cdots \\
0 & \cdots & 0 & h_{j_0}\alpha_3 & h_{j_0+1}\alpha_3 & h_{j_0+2}\alpha_3 & 0 & \cdots \\
0 & \cdots & 0 & h_{j_0}\alpha_4 & h_{j_0+1}\alpha_4 & h_{j_0+2}\alpha_4 & h_{j_0+3}\alpha_4 & \cdots \\
\vdots & & \vdots & \vdots & \vdots & \vdots & \vdots & \ddots\end{array}\right)
$$
for some $(\alpha_n)\in\mathbb{R}^{\mathbb{N}}$ such that
$(n\alpha_n)\in\ell^{s_{rq}}$ (note that the element $h_{j_0}\alpha_1$ is positioned at the $j_0$-th row
and the $j_0$-th column of the matrix).


\section{Domination by basis operators and representing operators}

As a result of the active research in several branches of the Harmonic Analysis, a lot of information is known about weighted norm inequalities for classical operators on weighted Banach function spaces, mainly regarding weighted $L^p$ and Lorentz spaces.
The bibliography on the subject is extremely broad; we refer the reader to  \cite{hardy-littlewood-polya} for the classical inequalities, and to \cite{ben,cru} and the references therein for an updated review of the state of the art. We will use also some  concrete results and ideas concerning weighted norm inequalities
 that can be found in the papers \cite{ash,kellogg,mhas,per}.

We will show in what follows the  characterization in terms of vector norm inequalities of what we call a representing operator for a Banach function space $X(\mu)$. This is essentially a modification of a basis operator $\mathcal F:L(\mu) \to \Lambda$ that allows to identify each function in $X(\mu)$ with some easy transformation of the basic coefficients of certain univocally associated function. Our motivation is given by the fact that, although the coefficients are not associated to a basis of the space $X(\mu)$, such kind of operator ---that we will call a representing operator---  allows to find an easy representation of the functions of the space by means of some basic-type coefficients. If $\mathcal F$ is a basis operator, we write $(\alpha_i(f))_{i =1}^\infty \in \Lambda$ for the basic coefficients of a function $f$, that is, $\mathcal F(f)=(\alpha_i(f))_{i =1}^\infty$.

\begin{definition}
Let $X(\mu)$ be a Banach function space over $\mu$ and $\ell$ a sequence space over the counting measure $c$ on $\mathbb N$. Let $\mathcal B$ be a Schauder basis of a Banach function space $L(\mu)$ and suppose that the basic coefficients of the functions of $L(\mu)$ are in a sequence space $\Lambda$ defined as the Banach lattice given generated by an unconditional basis of a Banach space.
Consider an operator $T: X(\mu) \to \ell$. We will say that $T$ is a \textit{representing operator}  for $X(\mu)$ (with respect to $\mathcal F$) if it is an injective two-sides-diagonal transformation of the basis operator $\mathcal F$.
\end{definition}

 Thus, technically a representing operator is an injective map such that there are a sequence $g=(g_j) \in (\Lambda)^{\ell}$ with $g_j \ne 0$ for all $j \in \mathbb N$ and a function $h \in X^{L}$, $h \ne 0$ $\mu$-a.e., such that for every
 $x \in X_1(\mu_1)$, the sequence $T(x)=(\beta(x)_j) \in \ell$ can be written as
$$
\beta(x)_j = P_j \circ T(x)= g_j \mathcal F (hx) = g_j \alpha_j(hx).
$$
That is, for  the elements $ y \in h \cdot X(\mu) \subseteq  L(\mu)$ we have that
$$
\alpha_j(y)=\mathcal F(y)= g_j^{-1} \beta(h^{-1} y)_j.
$$
Equivalently, for each $x \in X(\mu)$, there is a sequence $(\beta_j) \in \ell$ such that
$$
x =T^{-1}((\beta_j)) =h^{-1} \mathcal F^{-1}((g_j^{-1} \beta_j)).
$$

\begin{example} \label{ex1}
\noindent
\begin{itemize}

\item[(i)]
An easy example of the above introduced notion is the so called generalized Fourier series. Consider $p=2$, an interval $I$ of the real line, the space $L^2(I)$ endowed with Lebesgue measure $dx$ and a weight function $w:I \to \mathbb R^+$, $w >0$. Note that the multiplication operator $M_{w^{1/2}}:L^2(w dx) \to L^2(I)$ defines an isometry. Take a sequence of functions $(\phi_n)_n$ belonging to $L^2(w dx)$ and such that the associated sequence $(b_n)_n$, where $b_n=w^{1/2} \phi_n$ for all $n$ defines an orthonormal basis $\mathcal B$ in $L^2(I)$, that is, it is orthogonal, norm one and complete. Note that this is equivalent to say that it defines an orthonormal basis in the weighted space $L^2(w dx)$. Consider the Fourier operator $\mathcal F_{\mathcal B}$  associated to the basis $\mathcal B$ of $L^2(I)$. Then the operator $T: L^2(w dx) \to \ell^2$ given by $T:= id_{\ell^2} \circ \mathcal F_{\mathcal B} \circ M_{w^{1/2}}$ is a representing operator for $L^2(w dx)$.

Concrete examples of this situation are given by classical orthogonal basis of polynomials in weighted $L^2$-spaces. For example, for the trivial case of the weight equal to $1$ and the space $L^2[-1,1]$, we can define the functions $\phi_n$ to be the  Legendre polynomials, that are solutions to the Sturm-Liouville problem and define the corresponding Fourier-Legendre series. Other non trivial cases also for $I=(-1,1)$ are given by the weight functions  $w(x)= (1-x^2)^{-1/2}(x)$ and $w(x)=(1-x^2)^{ 1/2}(x)$ and the Chebyshev polynomial of the first and second kinds, respectively. Laguerre polynomials give other example for $I=(0,\infty)$ and weight function $w(x)= e^{-x}$.

\item[(ii)]
Take a function $0 < h \in X(\mu)^{L^p(\mu)}$ and consider a sequence $0 < \lambda = (\lambda_i) \in \ell^2$.
 Let us
write $c$ for the counting measure in $\mathbb N$. Consider
the space $\ell^1(\lambda c)= \{(\tau_i) : (\lambda_i \tau_i) \in \ell^1 \}$ with the corresponding norm $\| (\tau_i) \|_{\ell^1( \lambda c)} = \sum |\lambda_i \tau_i|$. Then we have that $\ell^2 \hookrightarrow \ell^1(\lambda c)$, and so the space of multiplication operators $(\ell^2)^{\ell^1(\lambda c)}$ is not trivial.
A direct computation shows also that  $(\ell^2)^{\ell^1(\lambda c)}=\ell^2(\lambda^2 c)$. Then, for every $\tau=(\tau_i) \in \ell^2(\lambda^2 c)$ with $\tau_i \neq 0$ for all $i \in \mathbb N$ we have that the operator $T: X(\mu) \to  \ell^1(\lambda c)$ given by $T(\cdot)= \tau  \mathcal F(h \, \cdot)$ is a representing operator for the space $X(\mu)$.

\end{itemize}

\end{example}

Let $J$ be a  finite subset of $\mathbb N$, and write $P_J:\ell \to \ell$ for the standard projection on the subspace generated by the elements of $\mathcal B$ with subindexes in $J$.
If $T:X(\mu) \to \ell$ is an operator, consider the net $\{P_J \circ T=:T_J: \mathbb N \supset J \, \text{finite}\}$, where the order is given by the inclusion of the set of subindexes, that is $P_J \circ T \le P_{J'} \circ T$ if and only if $J \subseteq J'.$ By definition,
$$
T = \lim_B P_B \circ T
$$
as a pointwise limit. In what follow we will characterize representing operators in terms of inequalities using this approximation procedure and a compactness argument. Thus, considering the basic (biorthognal) functionals $b'_i \in L(\mu)'$,  $i \in J$, associated to the basis of $L(\mu)$ that defines the Fourier operator that we are considering, we have
$$
P_J(x):= \sum_{j \in J} \langle f, b'_j \rangle e_j, \quad f \in L^p(\mu).
$$

Fix a function $h \in (X(\mu))^{L(\mu)}$ and suppose that $\Lambda^\ell$ is non-trivial. Assume that the conditions are given in order to obtain that $( \Lambda)^\ell = (\ell')^{\Lambda'} = (\Lambda' \pi \ell)^*$.
The domination  inequality that must be considered in this case is given by the following expression.
$$
\sum_{i=1}^n \int P_J \circ T(x_i) y_i' \, d c \le
\Big\| \Big( \sum_{i=1}^n  \sum_{j \in J} \langle h x_i, b'_i \rangle \langle e_j,  \cdot \,  y_i'  \rangle \Big) \Big\|_{\Lambda' \pi \ell}
$$
$$
=
\sup_{g \in B_{(\Lambda)^{\ell}} }
\Big( \sum_{i=1}^n  \sum_{j\in J} \langle h x_i, b'_i \rangle \langle e_j,  g y_i'  \rangle \Big),
$$
that is, we are considering the sequence $\big( \sum_{i=1}^n  \langle h x_i, b'_i \rangle (y_i')_j \big)_{j \in J} \in \Lambda' \pi \ell$ as the functional of the dual of  $( \Lambda)^\ell $  given by
$$
\Big( \sum_{i=1}^n  \sum_{j \in J} \langle h x_i, b'_i \rangle \langle e_j,  \cdot \,  y_i'  \rangle \Big):
 \Lambda^\ell  \to \mathbb R.
$$
After taking into account the particular descriptions of the elements of the spaces involved, we get the equivalent expression for the inequality
$$
\sum_{i=1}^n \sum_{j \in J} \langle T(x_i), e_j \rangle  \langle e_j ,((y_i')_j) \rangle
$$
$$
\le  \sup_{ g \in B_{\Lambda^{\ell}}}
\Big( \sum_{i=1}^n  \sum_{j \in J} \langle h x_i, b'_j \rangle \langle e_j,  (g_j ( y_i')_j)  \rangle \Big)
= \sup_{ g \in B_{\Lambda)^{\ell}}}
\Big(
\sum_{i=1}^n  \sum_{j \in J} \langle h x_i, b'_j \rangle g_j ( y_i')_j \Big)
$$

and so the initial inequality is equivalent to the following one,
$$
\sum_{i=1}^n \sum_{j \in J} T(x_i)_j \, (y_i')_j
\le
\sup_{g \in B_{(\Lambda)^{\ell}}}
\Big(
\sum_{i=1}^n  \sum_{j \in J} \alpha_j(h x_i) \, g_j \, ( y_i')_j
\Big),
$$
where $\alpha_j(h x_i)= \int h x_i b'_j d \mu$, $j \in J$, are the $j$-th Fourier coefficients of the function $h x_i$ associated to the basis $\mathcal B$.

Thus, the assumptions on the properties of $(X(\mu))^{L(\mu)}$ and $\Lambda^{\ell}$ provides the following


\begin{theorem}  \label{main}
Suppose that $\Lambda'$ and $\ell$ satisfies that $\Lambda'\pi \ell'$ is saturated and
$(\Lambda'\pi \ell')^*= (\Lambda')^\ell$, and let $h$ be a measurable function  such that $ 0 < |h| \in (X(\mu))^{L(\mu)}$.
The following statements are equivalent for an operator $T: X(\mu) \to \ell$.

\begin{itemize}

\item[(i)] For every finite set $J \subseteq \mathbb N$
the inequality
$$
\sum_{i=1}^n \sum_{j \in J} T(x_i)_j \, (y_i')_j
\le
C
\sup_{g \in B_{\Lambda^{\ell}}}
\Big(
\sum_{i=1}^n  \sum_{j \in J} \alpha_j(h x_i) \, g_j \, ( y_i')_j
\Big),
$$
holds for every $x_1,...,x_n \in X_1$ and $y_1',...,y_n' \in \ell'$.

\item[(ii)] $T$ is a representing operator with respect to $\mathcal F$, that is, there is a sequence  $g \in \Lambda^{\ell}$ such that
$(T(x))_j=g_j \cdot \alpha_j (hx)$ for all $x\in X(\mu)$ and $j \in \mathbb N$. In other words,  $T$ factors
through $\mathcal F$ as
 \begin{equation}\label{factodiag2}
\xymatrix{
X(\mu) \ar[rr]^{T} \ar@{.>}[d]_{M_h} & & \ell \\
L(\mu) \ar@{.>}[rr]^{\mathcal F}&  & \Lambda. \ar@{.>}[u]_{M_g}}
\end{equation}

\end{itemize}

\end{theorem}

\begin{proof}
Let us see that (i) implies (ii). We can assume without loss of generality that $C=1$. Note that as a consequence of Remark \ref{REM: StrongFactor} the requirements on $\Lambda$ and $\ell$ provides the conditions on these spaces for applying Corollary \ref{COR: StrongFactor}.
By the computations above, we obtain that for each finite set $J$ we have a norm one  sequence $g_J \in {\Lambda}^{\ell}$  satisfying that
$$
P_J \circ T(x) = g_J \cdot P_J \circ  \mathcal F(h x).
$$
Consider the net
 $\mathcal N:=\{g_J: J \subseteq \mathbb N \, \textit{finite} \}$, where the order is given by the inclusion of the finite sets used for the subindexes. We can assume without loss of generality that the support of each function $g_J$ is in $J$, that is, the coefficients $(g_J)_k$ of the sequence $g_J$ are $0$ for $k \notin J$.
Since all the functions of the net are in  the unit ball and due to the product compatibility of the pair defined by $\Lambda'$ and $\ell$, we have that the net is included in the weak* compact set  $B_{\Lambda^\ell}$. Therefore, it has a convergent subnet $\mathcal N_0$, that is, there is a sequence $g_0 \in B_{\Lambda^\ell}$ such that
$$
\lim_{\eta \in \mathcal N_0} g_\eta =g_0
$$
in the weak* topology given by the dual pair $\big\langle \Lambda' \pi \ell, \Lambda^\ell \big\rangle$.

Note now that for a fixed $x \in X(\mu)$, due to the fact that we are assuming that $\ell$ has a unconditional basis with associated projections $P_J$, we have
$$
\lim_{J \in \mathcal N \, \textit{finite} } P_J \circ T(x)= T(x).
$$
Then,
$$
T(x)=
\lim_{J \in \mathcal N \, \textit{finite} } P_J \circ T(x)= \lim_{\eta \in \mathcal N_0} P_\eta \circ T(x) =
\lim_{\eta \in \mathcal N_0} g_\eta \cdot \mathcal F(hx) = g_0 \cdot \mathcal F(hx).
$$
This gives (ii) and finishes the proof, since the converse holds by a direct computation.

\end{proof}

Let us provide an example.
Consider again Example \ref{ex1}(ii), and recall that
 $(\ell^2)^{\ell^1(\lambda c)}=\ell^2(\lambda^2 c)$.
Theorem \ref{main} gives that an injective operator $T:X(\mu) \to \ell^1(\lambda c)$ is a representing operator by means of the Fourier operator if and only if
for every finite set $J \subseteq \mathbb N$
the inequality
$$
\sum_{i=1}^n \sum_{j \in J} T(x_i)_j \, (y_i')_j
\le
C
\sup_{g \in B_{\ell^2(\lambda^2 c)} }
\Big(
\sum_{i=1}^n  \sum_{j \in J} \alpha_j(h x_i) \, g_j \, ( y_i')_j
\Big),
$$
holds for every $x_1,...,x_n \in X(\mu)$ and $y_1',...,y_n' \in \ell^1(\lambda c)'$.

\begin{remark}
Let us gives some sufficient conditions for the product sequence space appearing in Theorem \ref{main} to satisfy what is needed.
The product compatibility of the pair $\ell^{p'}$ and $\ell$ means that
$$
(\ell^{p'})^{\ell}= (\ell')^{\ell^p} = \big(\ell' \pi \ell^{p'} \big)^*.
$$
For example, if $\ell'$ is $p$-convex we have that $\ell' \pi
\ell^{p'} $ is saturated and so, a Banach function space
 (see Proposition 2.2 in \cite{san}). Moreover, the
quoted result provides also the equality (under the assumption of
saturation of the product)
$$
\big(\ell' \pi \ell^{p'} \big)' = (\ell^{p'})^{\ell}.
$$
Consequently, if the product is order continuous, we get the
desired result. Conditions under which this space is order
continuous are given in  Proposition 5.3 in
\cite{calabuig-delgado-sanchezperez}: for example, if the norm of
the product is equivalent to
$$
\| \lambda \|_\pi  \sim \inf \big\{ \|\eta\|_{\ell^{p'}}
\cdot \| \gamma\|_{\ell'}: | \lambda |= \eta \cdot \gamma,  \, \eta \in \ell^{p'}, \, \gamma \in \ell' \big\},
$$
the space is order continuous if $\ell'$ is assumed to be order continuous (recall that $p>1$ and so
$\ell^{p'}$ is order continuous too). The formula above for the product space works for example if $\ell$ is $p'$-concave,
since this implies that $\ell'$ is $p$-convex that together with the $p'$-convexity of $\ell^{p'}$ provides the result.
Concrete examples for $\ell^p$ spaces has been given in Example \ref{ex1}.

\end{remark}

\section{Operators associated to trigonometric series}


Relevant historical examples are the ones associated to the
Fourier series and the corresponding Fourier coefficients. We
finish the paper by explicitly writing the results presented
previously in this setting.  We will write $\hat{x}(\cdot)$ for
the $i$-th Fourier (real) coefficients of the function $x$ with
indexes in the set $\mathbb Z$, writing the coefficients $a_n$
asociated to $cos$ functions as $\hat{x}(i)$  with positive $i$
and the coefficients $b_n$ for the functions $sin$ as $\hat{x}(i)$
with negative $i$.

\begin{itemize}

\item Due to the Hausdorff-Young inequality, we know that for $1 <
p \le 2$, the Fourier transform $\mathcal F_p$--- sending $L^{p}[-
\pi, \pi] \to \ell^{p'}$ that assigns to each function the
sequence of its Fourier coefficients is well-defined and
continuous. The Fourier transform is defined as $\mathcal F_2: L^2
\to \ell^2$. Suppose that we want to check if a particular
operator $\mathcal G_2:L^2[- \pi, \pi] \to \ell^2$  can be
extended to $L^p[- \pi,  \pi]$ through $\mathcal F_p$. That is, is
there a factorization for $\mathcal G_2$ as
$$
\xymatrix{
L^{2}[- \pi, \pi] \ar[rr]^{\mathcal G_2} \ar@{.>}[d]_{i} & & \ell^{2}\\
L^{p}[- \pi, \pi] \ar@{.>}[rr]^{\mathcal F_p}&  & \ell^{p'}  \ar@{.>}[u]_{M_\lambda}}
$$
for the operator $\mathcal G_2$ for some multiplication operator
given by a sequence $\lambda$.

We have  shown that this is equivalent to the following inequalities to hold for the operator $\mathcal G_2$.
For each $x_1,..., x_n \in L^{2}[- \pi, \pi] $ and $\lambda_i \in \ell^2$,
$$
 \sum_{k=1}^\infty \sum_{i=1}^n (\mathcal G_2(x_i))_k (\lambda_i)_k
$$
$$
\le
C \big\| \big( \sum_{i=1}^n \hat{x_i}(k) (\lambda_i)_k \big) \big\|_{\ell^r}
=
C \Big( \sum_{k=1}^\infty \big| \sum_{i=1}^n \hat{x_i}(k) (\lambda_i)_k \big|^r \Big)^{1/r}.
$$

\item For $1 < p \le 2$ again, Kellogg proved an improvement of
the Hausdorff-Young inequality, that assures that the
corresponding Fourier coefficients of the functions in $L^p$ can
be found in the smaller mixed norm space $L^{p',2} \subseteq
\ell^{p'}$.    Fix $1 \le p,q \le \infty$. The mixed norm sequence
space  $L^{p,q}$ was defined in \cite{kellogg} as the space of
sequences $\lambda=(\lambda_k)_{-\infty}^\infty$ such that
$$
\|\lambda\| =  \Big( \sum_{m=\infty}^\infty \Big( \sum_{k \in I(m)} |\lambda_k|^p \Big)^{q/p} \Big)^{1/q} < \infty,
$$
where $I(m)=\{ k \in \mathbb Z: 2^{m-1} \le k \le 2^m \}$ if $m >0$, $I(0)=\{0\}$ and
$I(m)=\{ k  \in \mathbb Z: -2^{-m} \le k \le -2^{-m-1} \}$ if $m <0$. It is easy to see that $L^{p',2} \subseteq \ell^{p'}$,
and so we have a factorization for the Fourier map as
$$
\xymatrix{
L^{p}[- \pi, \pi] \ar[rr]^{\mathcal F_p} \ar@{.>}[d]_{i} & &   \ell^{p'} \\
L^{p}[- \pi, \pi] \ar@{.>}[rr]^{\mathcal K_p}&  & L^{p',2}.  \ar@{.>}[u]_{i}}
$$

In Theorem 1 of \cite{kellogg}, it is proved that the space of multiplication operators (multipliers)
from $L^{p',2}$ to $\ell^{p'}$ can in fact be identified with $\ell^\infty$.
Consequently, our results imply that
 for every finite set $J \subseteq \mathbb Z$
the inequality
$$
\sum_{i=1}^n \sum_{j \in J} (\mathcal F_p{x_i})_j\, (\lambda_i')_j
\le
C
\sup_{g \in B_{\ell^\infty}}
\Big(
\sum_{i=1}^n  \sum_{j \in J} \hat{x_i}(j) \, g_j \, ( \lambda_i')_j
\Big),
$$
holds for every $x_1,...,x_n \in X_1$ and
$\lambda_1',...,\lambda_n' \in \ell^p$, what is obvious. However,
note that this is essentially a characterization, since any other
operator $\mathcal G_p$ from $L^p$ and having values in a~sequence
space $\ell$ such that $(L^{p',2})^{\ell}= {\ell^\infty}$
satisfying these inequalities has to be of the form $g \cdot K_p$
for a certain sequence $g \in \ell^\infty$.

\item The Hardy-Littlewood inequality, also for $1 < p \le 2$,
provides an example of an operator $\mathcal H_p$  sending the
Fourier coefficients of the functions in $L^p$ to a weighted
$\ell^p$ space. For $1  < p < 2 $, consider the weighted  sequence
space $\ell^p(W)$, where the weight  $W$ is given by $W=(W_n)=
(1/(n+1)^{2-p})$. The Hardy-Littlewood inequality  can be
understood as the fact that the Fourier operator can be defined as
$\mathcal H_p: L^p[- \pi, \pi] \to \ell^p(W)$ (see
\cite[S.2]{ash}, in particular Theorem B). Note that the
multiplication operator $M_\gamma: \ell^p(W) \to \ell^p$ given by
the sequence $\gamma= \Big((1/(n+1)^{\frac{2-p}{p}}\Big)$ defines
an isometry. Therefore, the factorization scheme
$$
\xymatrix{
L^{p}[- \pi, \pi] \ar[rr]^{\mathcal \gamma \cdot \mathcal H_p} \ar@{.>}[d]_{i} & &   \ell^{p} \\
L^{p}[- \pi, \pi] \ar@{.>}[rr]^{\mathcal H_p}&  &  \ell^p(W)  \ar@{.>}[u]_{M_\gamma}}
$$
provides other example of the situation we are describing. Indeed, for every multiplication operator $_\tau$ for $\tau \in (\ell^p(W))^{\ell^p}$ we can give an operator $\tau \cdot \mathcal H_p$ satisfying this factorization. Our results implies the class of all these operators is characterized in the following way: if $T:L^{p}[- \pi, \pi] \to \ell^p$ satisfies the  inequalities
$$
\sum_{i=1}^n \sum_{j \in J} (T({x_i}))_j\, (\lambda_i')_j
\le
C
\sup_{g \in B_{(\ell^p(W))^{\ell^p}}}
\Big(
\sum_{i=1}^n  \sum_{j \in J} \hat{x_i}(j) \, g_j \, ( \lambda_i')_j
\Big),
$$
for each finite subset $J \subset \mathbb Z$, for every $x_1,...,x_n \in L^p[-2 \pi, 2 \pi]$
and $\lambda_1',...,\lambda_n' \in \ell^{p'}$, then it has a factorization as the one above
for a certain $\tau \in (\ell^p(W))^{\ell^p}$.

\item Let us recall Example \ref{ex1}(i). A representing operator
$T: L^2(w dx) \to \ell^2$ associated to a weight function $w$ and
an orthogonal basis $\mathcal B$ with respect to the corresponding
weight function was considered. It allowed a factorization as
$$
\xymatrix{
L^{2}(w dx) \ar[rr]^{T} \ar@{.>}[d]_{M_{w^{1/2}}} & &   \ell^{2} \\
L^{2}[I] \ar@{.>}[rr]^{\mathcal F_{\mathcal B}}&  &  \ell^2. \ar@{.>}[u]_{id}}
$$
The corresponding vector norm inequality characterizing this factorization is
$$
\sum_{i=1}^n \sum_{j \in J} (T({x_i}))_j\, (\lambda_i')_j
\le
C \Big\|
\sum_{i=1}^n  \sum_{j \in J} (T(x_i))_j \, ( \lambda_i')_j
 \Big\|_{\ell^1}
$$
$$
= C
 \sum_{j \in J} \Big| \sum_{i=1}^n   (\mathcal F_{\mathcal B}(x_i))_j  \, ( \lambda_i')_j \Big|
$$
for a given constant $C>0$ and for each finite subset $J \subset \mathbb N$,
for every $x_1,...,x_n \in L^2(w dx)$ and $\lambda_1',...,\lambda_n' \in \ell^{2}$.


\end{itemize}


\end{document}